\documentclass[a4paper,11pt]{amsart}
\usepackage{amssymb}
\usepackage[dvips]{graphicx}
\usepackage{amscd}
\newcommand{\comm}[1]{}

\def\dist{\operatorname{dist}}

\def\ti{\tilde}

\def\({\left(}
\def\){\right)}
\def\oli{\overline}
\def\uli{\underline}

\def\raw{\rightarrow}

\def\no={\neq}
\def\sm{\setminus}

\def\B{{\mathbb B}}
\def\C{{\mathbb C}}

\def\N{{\mathbb N}}

\def\P{{\mathbb P}}

\def\R{{\mathbb R}}

\def\NN{{\mathcal N}}
\def\OO{{\mathcal O}}
\def\PP{{\mathcal P}}

\def\RR{{\mathcal R}}
\def\SS{{\mathcal S}}

\def\al{\alpha}
\def\be{\beta}
\def\ga{\gamma}
\def\de{\delta}

\def\vep{\varepsilon}

\def\th{\theta}

\def\la{\lambda}

\def\om{\omega}

\def\De{\Delta}

\def\La{\Lambda}

\theoremstyle{plain}
\newtheorem{Main}{Theorem}

\newtheorem{Thm}{Theorem}[section]
\newtheorem{Prop}[Thm]{Proposition}
\newtheorem{Lem}[Thm]{Lemma}

\newtheorem*{Fact}{Fact}

\theoremstyle{remark}
\newtheorem{Rem}[Thm]{Remark}
\newtheorem{Def}[Thm]{Definition}

\begin{document}
\title[Perturbations of rational Misiurewicz maps]
{Perturbations of rational Misiurewicz maps}
\author{Magnus Aspenberg}

\thanks{The author gratefully acknowledges funding from the the Research Training Network CODY of the European Commission and the Swedish
  Research Council.
}

\begin{abstract}
In this paper we investigate the perturbation properties of rational
Misiurewicz maps, when the Julia set is the whole sphere (the other case is treated in
\cite{MA5}). In particular,
we show that if $f$ is a Misiurewicz map and not a flexible Latt\'es
map, then we can find a hyperbolic map arbitrarily close to $f$.
\end{abstract}

\maketitle

\section{Introduction}

The notion of Misiurewicz maps has its origin from the paper \cite{MM}
by M. Misiurewicz. In honor of this paper, we proceed with the
following definition. First, let $J(f)$ be the Julia set of $f$, $F(f)$ the Fatou set of $f$ and $Crit(f)$ the critical set of $f$. Let $\om(c)$ be the
omega limit set of $c$.

\begin{Def}
A non-hyperbolic rational map $f$ without parabolic periodic points
satisfies the {\em Misiurewicz condition} if for every $c \in Crit(f) \cap J(f)$,
we have $\om(c) \cap Crit(f) = \emptyset$.
\end{Def}


In \cite{McM}, McMullen showed that small Mandelbrot copies are dense in the bifurcation locus in any (non-trivial) analytic family of rational maps.
An important element of the proof of this involves the fact that
Misiurewicz maps create a family of polynomial-like maps nearby. From
this stems the Mandelbrot copies.
In this paper we consider rational Misiurewicz maps for which the Julia set
is the whole sphere and show that such maps can be perturbed into a
rich family of postcritically finite maps. One main ingredient is that we get
control of all critical points, which enables perturbations into
hyperbolic maps, for instance.
We use the the supremum norm on the Riemann sphere to calculate the distance
between two rational funtions. The main result of this paper is the following.
\begin{Main} \label{cluster}
Let $f$ be a rational Misiurewicz map and not a flexible Latt\'es
map. Then we can find a hyperbolic rational map $g$ arbitrarily close to $f$.

Moreover, if in addition $J(f) = \hat{\C}$, then $g$ can be chosen so that all critical points lie in superattracting cycles.
\end{Main}
The case when $J(f) \neq \hat{\C}$ is treated in \cite{MA5}. In that paper it is shown that Misiurewicz maps for which $J(f) \neq \hat{\C}$ are in fact Lebesgue density points of hyperbolic maps. We therefore assume that $J(f) = \hat{C}$ to prove the first part of Theorem \ref{cluster}. In fact, we will only focus on the second part (which of course implies the first in the case $J(f) = \hat{\C}$).

The proof of Theorem \ref{cluster} works as long as the given
Misiurewicz map admits no quasiconformal deformations. This is
why the flexible Latt\'es maps have to be excluded, because they are the
only Misiurewicz maps which admit quasiconformal deformations.

It will be clear form the construction that in fact $f$ (having $J(f)=\C$) can be perturbed into any post-critically finite map $g$ with prescribed repelling orbits on which the critial points shall land on.


The parameter space $\RR^d$ of rational maps of degree $d$ is a
$2d+1$-dimensional complex manifold. We will mostly consider {\em
  normalised} families of rational maps, which are $\RR^d$
modulo conjugacy classes of M\"obius transformations. The space $[
{\RR}^d ]$ of normalised
rational maps of degree $d \geq 2$ has dimension equal to $2d-2$.



Misiurewicz maps have good expansion properties, by a Theorem by
Ma\~n\'e \cite{RM}. In particular, they satisfy the so called Collet-Eckmann condition, defined as follows.
\begin{Def}
A rational map $f$ satisfies the {\em Collet-Eckmann} condition (CE) if there are constants $C > 0$ and $\ga > 0$ such that for every critical point $c \in J(f)$, not containing any other critical point in its forward orbit, we have
\[
|(f^n)'(fc)|\geq C e^{\ga n},
\]
for all $n \geq 0$.
\end{Def}
Recently Rivera-Letelier \cite{RL}
showed that one can perturb a so called backward contractive function to
obtain a Misiurewicz map, provided the Julia set is not the whole
Riemann sphere. The backward contraction condition is weaker than for
example the Collet-Eckmann condition. Hence as a consequence of
Rivera-Letelier and Theorem \ref{cluster} (or rather \cite{MA5}),
every Collet-Eckmann map for which the Julia set
is not the whole Riemann sphere can be approximated by a hyperbolic
map. Combining this with \cite{MA} (revised version) we obtain the
following characterization of rational Misiurewicz maps:

Let $f$ be a rational Misiurewicz map. Then if $f$ is not a flexible
Latt\'es map it can be approximated by a hyperbolic map. Moreover,

\begin{itemize}

\item if $J(f) = \hat{\C}$, then $f$ is a Lebesgue density point of CE-maps,

\item if $J(f) \neq \hat{\C}$, then $f$ is a Lebesgue density point of
  hyperbolic maps.

\end{itemize}

In the first case above, the CE-maps have their Julia set equal to the
whole sphere (which is also a consequence of \cite{MA}).
In view of the Fatou conjecture, flexible Latt\'es maps should also be approximable by hyperbolic maps but we have no proof. It seems this
question is not much simpler than the Fatou conjecture itself. 



\subsection*{Acknowledgements} Firstly, I am grateful to Jacek Graczyk
for encouraging me to write this down and for many helpful comments
and interesting discussions. I thank Michael Benedicks and Nicolae
Mihalache for many interesting discussions and useful remarks on a preliminary version. I owe my thanks to Detlef M\"uller for introducing me to the theory of resolution of singularities and to Mattias Jonsson for further helpful remarks and clarifications. This helped me tackle the problem on higher order critial points.

A first version of this paper was written at Laboratoire de
Math\'ematiques, Universit\'e Paris-Sud. The current revised version
was written at Christian-Albrechts Universit\"at zu Kiel. I gratefully
acknowledge the hospitality of both these departments.

\section{Preliminaries}

Put $f=f_0$ and assume that $f$ is a Misiurewicz map. We assume that
$J(f)= \hat{\C}$ unless otherwise stated.
The idea is to start with a family $f_a$ of rational maps
parameterized by a small disk $\B(0,r) \subset [ {\RR}^d ]$, or radius $r > 0$,
where $f_0$ corresponds to the parameter $a=0 \in \B(0,r)$. We then study
the iterates of the corresponding critical points for parameters
in certain so called Whitney subdisks in $\B(0,r)$ (see definition below). We will also consider $1$-dimensional disks $B(
0,r) \subset \B(0,r)$, where $\B(0,r)$ is $(2d-2)$-dimensional. Let $c_j(a)$ be the set of critical points
for $f_a$. Define, for any $a \in \B(0,r)$,
\begin{equation} \label{xinj}
\xi_{n,j}(a)=f_a^n(c_j(a))=f^n(c_j(a),a).
\end{equation}
\subsection{Higher order critical points} \label{splitting}
Of course some critical
point might split under perturbation and in this case it is not clear
how to define the functions $c_j(a)$ in (\ref{xinj}). In this case one cannot generally hope that the critical points move
analytically. If higher order critical points occur
we reparameterise the family using a theorem on the resolution
of singularities, which goes back to Hironaka's theorem
on the resolution of singularities \cite{Hironaka}. In fact, what we are looking for is
a complex analytic version of Lemma 9.4 in \cite{Muller-Ricci}. Based
on \cite{Atiyah}, it is shown that critical points move analytically in a ``lifted parameter
 space'' for a real
polynomial where the coefficients are real-analytic. Since a direct reference
to the complex analytic version of this result was not found, although
it seems to be a wellknown folklore result, we outline the argument
based on \cite{Muller-Ricci} and \cite{Atiyah}.

Put $f'(z,a)=G(z,a)$ and consider the analytic set $\{ G(z,a)=0 \}$.
This can be viewed as a real-analytic set of
double dimension. The singularities of this set appear precisely at the
points where $f$ has a critical point of higher order. To resolve
these singularities, we first write $G=u + iv$, where $u=u(x,y,a)$ and
$v=v(x,y,a)$ are real-analytic in $x,y$ and
$a=(a_1,\ldots,a_{4d-4})$. By the Resolution Theorem in \cite{Atiyah}
we get an analytic map $\phi_u : (y_1,\ldots,y_n) \raw
(x,y,a_1,\ldots,a_{n-2})$, such that
\begin{align}
u \circ \phi_u &= \vep_u \prod_{i=1}^n y_i^{k_{i,u}}, \nonumber \\
\end{align}
where $\vep_u$ is non zero and $n=4d-2$.

The set $\{ u=0 \}$ has (real) codimension $1$
and therefore there must be some $k_{i,u}, > 0$. If we
restrict to the set $B_{i} = \{ y_i = 0 \}$ then
for $y \in B_ i$ we have $u \circ \phi_u (y) = 0$.
Let us identify $B_{i}$ locally with the $n-1$ (real) dimensional
ball $\B^{n-1}$.

Now, consider the function $F(y)= v(\phi_u(y))$, for $y  \in
\B^{n-1}$. Using the Resolution Theorem again, we get an analytic
function $\phi_v : (x_1,\ldots,x_{n-1}) \raw \B^{n-1}$ such that
\[
F \circ \phi_v = \vep_v \prod_{i=1}^{n-1} x_i^{k_{i,v}}.
\]

Some $k_{j,v} > 0$, so (locally) identifying the set $B_j = \{ x_j = 0\}$ with
the unit ball $\B^{n-2}$ (of real dimension $n-2$) we get
$v(\phi_u(\phi_v(x))) = 0$ for $x \in \B^{n-2}$. Since $\phi_v(x) \in
\B^{n-1}$ we have, for $x \in \B^{n-2}$,
\begin{align}
u \circ \phi_u \circ \phi_v (x) &= 0 \nonumber \\
v \circ \phi_u \circ \phi_v (x) &= 0 \nonumber.
\end{align}

Put $\phi = \phi_u \circ \phi_v$. Then we see that $G \circ \phi (x)
=0$, for $x \in \B^{n-2}$. Let $\phi (x) = (\psi_0(x), \psi_1(x)) \in
\R^2 \times \R^{n-2}$. Then $\psi_0(x)$ parameterises one critical
point. For each choice of $B_i,B_j$ above we get, in this way, a
parameterisation of each critical point.
For fixed $x$, the function $G(z,\psi_1 (x)) $ has a zero in
$c_1=c_1(x)$ (which must be a critical point for $f(z,\psi_1(x))$).
If we view $G$ as a complex function again, we can write
\[
G(z,\psi_1(x)) = (z-c_1(x)) G_1(z,x),
\]
where $z_1$ and $G_1$ are real-analytic in $x$ and
$G_1$ is complex analytic in $z$.

We apply the same argument to $G_1$ and obtain a new parameterisation
(proper analytic map) $\psi_2: x' \raw x $ such that
\[
G_1(z, \psi_2(x')) = (z-c_2(x')) G_2(z,x').
\]
Then
\[
G(z,\psi_1 \circ \psi_2 (x')) = (z - c_1 ( \psi_2(x'))) (z - c_2(x')) G_2(z,x').
\]

Continuing in the same manner we obtain a proper
analytic map $\psi$, which is a composition of finitely many
proper analytic maps $\psi_1, \psi_2, \ldots$ constructed above, such that
\[
G(z,\phi (a)) = E(z,a) (z-c_1(a)) \cdot \ldots \cdot (z-c_{2d-2}(a)),
\]
where $E$ is non-vanishing and where $c_j=c_j(a)$ are real-analytic,
$a \in \B^{n-2}$.

We have to prove that in fact $c_j(a)$ are complex-analytic.
Let $K$ be the set of parameters
in the original parameter space such that there are exactly $2d-2$
distinct critical points. Then since $\psi$ is proper,
$K' = \psi^{-1}(K)$ has full measure in $\B^{n-2}$. Since the critical
points are simple in $K$, we have
$2d-2$ analytic functions $c_j$ on $\B^{n-2}$ which are analytic on $K$
and continuous on $\B^{n-2}$. Since $K$ has full measure,
the functions $c_j$ can be continued
analytically to the whole set $\B^{n-2}$ by Rado's Theorem (p. 301-302 in
\cite{Chirka}).
Hence the critical points $c_j$ move analytically as functions of the
new variables in $\B^{n-2}$.

However, the new space $\B^{n-2}$ is not necessarily normalised in the
sense described in the Introduction; the function $\psi$ need not be
injective. This may cause some problems regarding transversality. We
will deal with this in the section on Transversality.


\subsection{Holomorphic motions and the parameter functions $x_j$}

Let us state the following Theorems by Ma\~n\'e \cite{RM}:
\begin{Thm}[Ma\~n\'e's Theorem I] \label{mane}
Let $f:\hat{\C} \mapsto \hat{\C}$ be a rational map and $\La \subset J(f)$ a
compact invariant set not containing critical points or parabolic points. Then either $\La$ is a hyperbolic set or $\La \cap \omega(c) \neq \emptyset$
for some recurrent critical point $c$ of $f$.
\end{Thm}
\begin{Thm}[Ma\~n\'e's Theorem II] \label{mane2}
If $x \in J(f)$ is not a parabolic periodic point and does not intersect $\om(c)$ for some recurrent critical point $c$, then for every $\vep > 0$, there is a neighborhood $U$ of $x$ such that

\begin{itemize}
\item For all $n \geq 0$, every connected component of $f^{-n}(U)$ has diameter $\leq \vep$.

\item There exists $N > 0$ such that for all $n \geq 0$ and every connected component $V$ of $f^{-n}(U)$, the degree of $f^n |_V$ is $\leq N$.

\item For all $\vep_1 > 0$ there exists $n_0 > 0$, such that every connected component of $f^{-n}(U)$, with $n \geq n_0$, has diameter $\leq \vep_1$.

\end{itemize}
\end{Thm}
An alternative proof of Ma\~n\'e's Theorem can also be found by L. Tan and M. Shishikura in \cite{ST}.
Let us also note that a corollary of Ma\~n\'e's Theorem II is that a Misiurewicz map cannot have any Siegel disks, Herman rings or Cremer points (see \cite{RM} or \cite{ST}).

Since $f$ is a Misiurewicz map, there is some $k \geq 0$ such that the set
\[
P^k(f) = \oli{\bigcup_{n > k, c \in Crit(f) \cap J(f)} f^n(c)}
\]
is compact, forward invariant and does not contain any critical or
parabolic points. Let us put $\La = P^k(f)$ for the smallest such $k$.
By Ma\~n\'e's Theorem I it follows that $\La$ is a hyperbolic set. Hence, there exists a holomorphic motion $h: \B(0,r) \times \La \raw \hat{\C}$, such that $h_a$ is an injection for each $a \in \B(0,r)$ and $h_a: \La \raw \La_a$, where $\La_0=\La$ and the following holds for $z \in \La$:
\[
f_a \circ h_a (z) = h_a \circ f_0(z).
\]
Put $v_j(a) = f_a(c_j(a))$ for each marked critical point $c_j(a)$ and $v_j(0)=v_j$.

For $a \in \B(0,r)$, let us introduce the {\em parameter functions} $x_j$
\[
x_j(a) = v_j(a) - h_a (v(0)).
\]
Let us define $\mu_{n,j}(a)=h_a(f_0^n(v_j(0)))$.

\begin{Def}
Given a set $E$ and $\de > 0$, we call the set $\{ x : dist(x,E) \leq \de, \}$, a $\de$-neighbourhood of $E$.
\end{Def}

\subsection{Some constants} \label{const}
We define the constant $\la$ to the minimum over $|f_a'(z)|$ over all $(z,a) \in \La_a \times
\oli{B}(0,r)$. Since $\La=\La_0$ is hyperbolic, in general there is some $N > 0$ such that $|(f^N)'(z,a)| \geq \la_0 > 1$ for all $(z,a) \in \La_a \times \oli{B}(0,r)$ for some small $r > 0$, but for simplicity assume that $N=1$ so that $\la > 1$.

Let $\NN$ be a $\de'$-neighbourhood around
$\La$, such that $|f_a'(z)| \geq \la/2$, for all $z
\in \NN$ and $a \in \B(0,r)$. Moreover,
let $U$ be a $\de$-neighbourhood around the critical points of $f$ so
that $U \cap \NN = \emptyset$ and $0 < \de < 1/2$. Choose $r > 0$ such
that every $c_j(a)$ belongs to a $\de^{10}$-neighbourhood around
$Crit(f_0)$ for every $a \in \B(0,r)$. Moreover, let $U_l$ be a
$\de_l$-neighbourhood around $Crit(f)$, for some $\de_l \leq
\de$. These $U_l$ will be defined inductively later.

\section{Transversality}

We will in this section study the functions $x_j(a)$. First,
note that $x_j(0)=0$ for all $j$. We cannot have $x_j(a) \equiv 0$
for all $j$ by Theorem E in \cite{MSS} (see also Theorem A in
\cite{MA3}), because then all maps in $\B(0,r)$ would be Misiurewicz
maps. Hence we may assume that at least one $j$ has that $x_j(a)$ is
not identically equal to zero, i.e. $x_j$ has {\em finite order contact} (see definition below).
In $\B(0,r)$ the function $x_j$ is an analytic function in several
variables and has a power series expansion
\begin{equation} \label{x-series}
x_j(a) = \sum_{|\al| \geq 1} b_{\al}a^{\al}
\end{equation}
where $\al=(\al_1,\ldots,\al_{2d-2})$ is a multi-index, $\al_j \geq 0$
and $|\al| = \sum \al_j$.


\begin{Def}
If $B(0,r)$ is a $1$-dimensional disk we say that the critical point $c_j(a)$ has contact of order $k$ if
\[
x_j(a) = v_j(a) - h_a (v(0)) = K_1a^K + \dots,
\]
for some $K_1 \neq 0$. If $x_j(a)$ is identically equal to zero we say
that $c_j(a)$ has contact of infinite order in $B(0,r)$.

If we consider $x_j(a)$ as a function of $a \in \B \subset \B(0,r)$
(where $\B$ is an open set of dimension $\leq 2d-2$),
then we say that $c_j$ has finite order contact in $\B$ if $x_j$ is
not identically equal to zero in $\B$.
\end{Def}

If it is evident in which set $x_j$ has finite order contact in, we just say
$x_j$ has finite order contact.
Dropping the index and writing $x(a),v(a)$ we mean $x_j(a),v_j(a)$
respectively for some index $j$. This index is chosen so that $x_j(a)$
has finite order contact unless otherwise stated. Also, we write
$v_a=v(a)$.

\subsection{Tangent cones}
We want to restrict to parameters where $x(a) \neq 0$ and $x'(a) \neq
0$.
To this end, we construct a cone-like set of the form $V_0 \times \C$,
where $V_0 \subset \P(\C^{2d-3})$ is an open ball of directions, such that
$x'(a) \neq 0$ and $x(a) \neq 0$ for all $a \in V_0 \times
B(0,r) \sm \{0\}$ and for all $x(a)=x_i(a)$ which have finite oder contact in
$\B(0,r)$ (here $B(0,r) \subset \C$). Let us assume that $x_i(a)$ has
finite order contact for the set of indices $i \in I$. Each equation
from the the set of equations
\begin{align}
x_i'(a) &= 0, \quad i \in I \nonumber \\
x_i(a) &= 0, \quad i \in I \nonumber
\end{align}
defines an analytic set. Let us number these analytic sets as
$A_1,A_2, \ldots, A_n$. We now use the standard theory of
analytic sets (see e.g. \cite{Chirka}); if $A$ is an analytic set,
then the tangent cone
$C(A,0) \subset \C^{2d-2}$ to $A$ at $0$ is defined as set of vectors $v \in \C^{2d-2}$ such that
there exists $a_j \raw 0$, $a_j \in A$ and real numbers $t_j > 0$ such
that $t_j a_j \raw v$ as $j \raw \infty$. By \cite{Chirka} p.83, the set $C(A,0)$ is an
algebraic subset in $\C^{2d-2}$. Let $p(v)$ be the projection of
$\C^{2d-2}$ onto $\P(\C^{2d-3})$. Now, applying this to the sets $A_j$, we see that since $p(C(A_j,0)) \neq \P(\C^{2d-3})$ and $C(A_j,0)$
is closed there must be an open ball $V_0 \subset p( \cap_j C(A_j,0) )^c$.

Moreover, any $1$-dimensional disk $B(0,r) \subset \B(0,r)$ is determined by
a direction vector $v \in \P(\C^{2d-3})$. Let us pick a for $v$ representative direction vector  $v_0=(\al_1,\al_2,\ldots,\al_{2d-2})$ (i.e. so that $v_0$ has direction $v$). Then
the plane in which $B(0,r)$ lies can be
parameterised by $$\{a \in \C^{2d-2}: a=(\al_1 t, \al_2 t, \ldots, \al_{2d-2}t), t \in \C \}.$$
The expansion of $\th$ in $B(0,r)$, where $\th$ is either some $x_j(a)$ or a component of $x_j'(a)$, becomes
\begin{equation} \label{theta}
\th(t) = \sum_{k \geq k_0} p_k t^k + \ldots,
\end{equation}
where each $p_k=p_k(\al_1,\ldots,\al_{2d-2})$ is a polynomial in the variables $(\al_1,\ldots,\al_{2d-2})$. By Proposition 1 p. 83 in \cite{Chirka}, we have $\{ p_{k_0} = 0\} = C(A_j,0)$. Since $\cap_j C(A_j,0) \supset  C(\cap_j A_j,0)$ we have that for any $1$-dimensional disk $B(0,r)$ with direction $v \in V_0$, the expansion of any $x_j$ or component of $x_j'$ becomes
\begin{equation} \label{onevarexp}
\th(t) = K_1 t^k + \ldots,
\end{equation}
where $K_1$ is continuous function of the direction and $k$ is
constant for all $v \in V_0$. We say that $k$ is the order of $f$ in the cone $W = V_0 \times B(0,r)$. This $W$ is the starting {\em
  good cone} where the $x_j$ behaves nicely. The importance of
$W$ is that $x_j$ has bounded distortion on so called dyadic disks
defined as follows.
\begin{Def}
A disk $D \subset \R^n$ of dimension $\leq n$ is called a {\em $k$-Whitney disk } or simply a {\em
 Whitney disk} if
\begin{equation} \label{dyadic}
diam (D) \geq k dist(D,0).
\end{equation}
\end{Def}

We say that $x'$ has bounded distortion in some set $D$ if
\begin{equation} \label{xprim}
\| x'(a) - x'(b) \| \leq \vep \|x'(a) \|,
\end{equation}
for all $a,b \in D$, for some small (fixed) $\vep > 0$.

By compactness of $\oli{V}_1$ the function $K_1$ is uniformly continuous in $V_1$. The following lemma now follows easily from (\ref{onevarexp}).
\begin{Lem} \label{circles}
There is some $0 < k < 1$ such that every $x_j(a)$ and every component of $x_j'(a)$ has bounded distortion on $k$-Whitney disks in $W$.
\end{Lem}

Moreover, we will use the following folklore result which follows from
\cite{SvS} by S. van Strien (see also \cite{AE} by A. Epstein). In \cite{SvS}
the definition of Misiurewicz maps differs. However,
in the case where the Julia set is the whole sphere then our definition coincides with that of van Strien.

\begin{Thm} \label{ttrans}
If $\B(0,r)$ is a $2d-2$ dimensional ball of normalised rational maps,
then the function $G: a \raw (x_1(a), \ldots,x_{n}(a))$ is a local
immersion unless $f$ is a flexible Latt\'es map.
\end{Thm}
Hence if $f$ is not a flexible Latt\'es map, the tangent
vectors $x_j'(0)$ are all linearly
independent. By continuity this also holds in the some small parameter
ball $\B(0,r)$.

If critical points of higher order exist for $f$, then we change to the
new variables described in Subsection \ref{splitting}. Although it is
clear that in this new space $\B^{n-2}$, $a=0$ is still an isolated
zero of $G$ (because $\psi$ is proper), the new space may not be normalised.

In this case we will perturb our original function $f$ into a new
Misiurewicz map, for which every critical point is simple, and still
the Julia set is the whole sphere. To do this, let $c$ be a higher
order critical point and assume that $c=c_1(0)= \ldots = c_N(0)$, where
the functions $c_j$ are analytic in $\B^{n-2}$, $1 < N \leq 2d-2$
(i.e. $c$ has order $N$).

For $1 \leq i \leq N$,
consider the set $Z_i = \{ a \in \B^{n-2} : x_j(a)=0, \text{ for all $j \neq
  i$}  \}$. By Lemma 2.3 in \cite{SvS}, there is a locally univalent
parameterisation of $Z_i$:
\[
\Psi_i: \C \raw Z_i,
\]
where $\Psi_i(0)=0 \in \B^{n-2}$.
Following \cite{SvS} (Lemma 3.1) we note that there
is some periodic repelling point $q_i(\la)$ arbitrarily close to the
critical value $v_i(0)=f^{k+1}(c_i(0),0) \in \La$. For $\la \in \C$,
put
\[
X_i(\la) = f_{\Psi(\la)}^{k+1}(c_i(\la)) - q_i(\la).
\]
We now note note that if $(x_i  \circ \Psi_i)'(0) \neq 0$, for all
$i$, then the tangent vectors $x_i'(0)$ are linearly
independent and $G$ is an immersion as in Theorem \ref{ttrans}. In
this case we do not need to change our function $f$.

Hence suppose that  $(x_i  \circ \Psi_i)'(0) = 0$. That
means that $(x_i  \circ \Psi_i)(\la) = g(\la) \la^d$, where $d \geq 2$
and $g(\la) \neq 0$ in a neighbourhood $B(0,\de)$ of $0$. We claim that
$X_i \circ \Psi_i$ has $d$ zeros in $B(0,\de)$ if $q_i=q_i(0)$ is
chosen sufficiently close to $v_i(0)$. The proof of this claim is
identical to that of \cite{SvS}, pp. 46-48. With $U_i(\la) = q_i(\la)
- v_i(\la)$ we get
\[
t^d g(t) + U_i \circ \Psi_i(t) = X_i \circ \Psi_i(t).
\]
Hence if $u_i = U_i \circ \Psi_i$ is sufficiently close to zero
(i.e. $q_i$ is sufficiently close to $v_i$), the image of $\partial B(0,\de)$ under the function $t^d \frac{g(t)}{u(t)}$ is a curve that encircles $-1$ $d$ times. By the Argument Principle there is at least one solution $t_0 \in B(0,\de)$ to
\[
t^d \frac{g(t)}{u(t)} = -1.
\]
For this $t_0$, the function $f_{\Psi_i(t_0)}$ is a new Misiurewicz
map different from $f_0$ and where $c_i(\Psi(t_0))$ is simple. The other critical
points $c_j(\Psi(t_0))$, $j \neq i$ still coincide and form a critical
point of order $N-1$.

Repeating this argument $N-1$ times we obtain a
Misiurewicz map such that all critical points are simple.

\subsection{Outline of proof of Theorem \ref{cluster}}
To prove the main result we take a full-dimensional Whitney parameter disk $\B_0 \subset W$. Suppose that $c_1$ has finite orde contact in $\B_0$.
We first show that the set $\xi_{n,1}(\B_0)$
grows up to some definite size before it leaves $\NN$ (Lemma \ref{initdist}). By bounded distortion
the set $\xi_{n,1}(\B_0)$ will contain a disk of diameter at least
$S=S_0$, where $S$ is some ``large scale''.

By
normality and compactness we show that for some $m$,
$\xi_{n+m,1}(B(0,r)) $ covers the whole Riemann sphere apart from at
most $2$ points. In particular, $c_1(a)$ and two of its pre-images
cannot be omitted by $\xi_{n+m,1}(B(0,r))$. Hence $c_1(B(0,r)) \subset
\xi_{n+m,1}(B(0,r))$ and there is a solution to
$\xi_{n+m,1}(a)-c_1(a)=0$ in $B(0,r) \subset \B(0,r)$. Next we will pass to the analytic subset $\B_0 \subset W$ where
$\xi_{n+m,1}(a)-c_1(a)=0$, and argue inductively.


In the next step, we consider $c_2$, which also has to have finite
order contact. In fact, assume that there are no more critical points
than $c_1$ of finite order contact. Then there would be a small ball $\B$
around the solution $a$ to $\xi_{n+m,1}(a)-c_1(a)=0$, where $\B$ is a
family of Misiurewicz maps. This is impossible by \cite{MA3}.

Now, we try to connect this $c_2$ with itself also in the same way as
we did with $c_1$.
However, to continue we need to have control of the shape and size of $\B_1$.
These results are mainly dealt with in the sections \ref{Distortion}
and \ref{Closing}. We show
that we have good control of the geometry of $\B_1$ in Whitney disks,
so that the
parameter function $x_2$ maps $\B_1$ onto small circles, so that
$\xi_{n,2}(\B_1)$
in turn grows to another large scale $S_1$ (which is typically less than $S_0$).
Then, as in the previous case for $c_1$, we can use non-normality and
compactness again to
get another $N_1$ for which $\xi_{n+m,2}(a)-c_2(a)=0$, for some $a \in \B_1$ and
$m \leq N_1$. A new manifold $\B_2 \subset \B_1$ is thereby formed,
where $\xi_{n+m,2}(a)-c_2(a)=0$.
We continue in this manner until all critical points are in the Fatou set.

\comm{
To this end, one of the main tools is strong distortion estimates for
Whitney disks in parameter space. More precisely, given the cone $W \subset \B(0,r)$ we
will consider smaller open {\em Whitney disks} $D_0=\B(a_0,r_0)$, $D_0
\subset W$ where $r_0/|a_0| = k_0$, for some $0 < k_0 < 1$. We say that $D_0$ is a {\em $k_0$-Whitney disk}. One can also define Whitney disks with lower dimension than $2d-2$ with the analoguous defition.

We show that the set $\B_0$ contains a $1$-dimensional $k_0$-Whitney disk $D_0$ and study the orbits $\xi_{n,j}(D_0)$ for some $j \in \{1,\ldots,2d-2\}$.
Then we let $\xi_{n,j}(D_0)$ grow exponentially under bounded
distortion up to the large scale. After this we use a non-normality
argument to get that for some finite $m \leq N$ (where $N$ is a constant), $\xi_{n+m,j}(D_0)$ covers some $\de_0$-neighbourhood
of the critical points (here we use that the transversal critical
point $c_j(a)$ belongs to the Julia set $J(f)$ for some $a$ "close to the center" of $\B_0$). For some (smallest) $m$ we have $\xi_{n+m,j}(D_0) \supset c_{j}(D_0)$. From this we deduce
that there is solution in $D_0$ to the equation
$\xi_{n+m,j}(a)=c_{j}(a)$. After this, a new submanifold $\B_1 \subset \B_0$ is found in which $\xi_{n+m,j}(a)=c_{j}(a)$ and we repeat the argument until all critical points are in the Fatou set.
}

\section{Distortion lemmas} \label{Distortion}

In this section we state the necessary distortion lemmas that will be
needed to get a Whitney parameter disk to grow to the large scale
before it leaves $\NN$. Many lemmas in this section are proven in
\cite{MA3} (and also in \cite{MA}) in a one-dimensional version.
In this section assume always that $c=c_j$ is transversal, i.e. $x(a)=x_j(a)=c_j(a)-v_j(a)$ is not identically equal to zero. Recall the notation $v_j(a)=v_a$ for the critical values.

Let us start with the following lemma (see \cite{WR}).
\begin{Lem} \label{prod-dist}
Let $u_n \in \C$ be complex numbers for $1 \leq n \leq N$. Then
\begin{equation}
\biggl| \prod_{n=1}^N (1+u_n)-1 \biggr| \leq
\exp \biggl(\sum_{n=1}^N |u_n| \biggr) - 1.
\label{ineq-1}
\end{equation}
\end{Lem}
The following lemma is a modified version of Lemma 3.2 in \cite{MA3}.
\begin{Lem}[Main Distortion Lemma] \label{main-dist}
For each $\vep > 0$ there exists an $r > 0$ and $\de' > 0$ such that the following holds. Let $a,b \in W$.
Then as long as $f_a^j(v_a),f_b^j(v_b) \in \NN$ (recall that $\NN$ depends on $\de'$), for all $0 \leq j \leq n$ we have
\[
\biggl| \frac{(f_a^n)'(v_a)}{(f_b^n)'(v_b)} - 1 \biggr| \leq \vep.
\]
\end{Lem}
The same statement holds if one replaces $v(s)=\xi_0(s)$, $s=a,b$, by $\mu_0(t)$, $t=a,b$.
\begin{proof}
The proof goes in two steps. Let us first show that
\begin{equation}
\biggl| \frac{(f_t^n)'(\mu_0(t))}{(f_t^n)'(\xi_0(t))} - 1 \biggr| \leq \vep_1,
\label{close-to-1}
\end{equation}
where $\vep_1=\vep(\de')$ is close to $0$. We have
\begin{align}
\sum_{j=0}^{n-1} \biggl| \frac{f_t'(\mu_j(t)) - f_t'(\xi_j(t))}{f_t'(\xi_j(t))} \biggr| &\leq C \sum_{j=0}^{n-1} |f_t'(\mu_j(t)) - f_t'(\xi_j(t))| \nonumber \\
&\leq C \sum_{j=0}^{n-1} |\mu_j(t)-\xi_j(t)| \nonumber \\
&\leq C \sum_{j=0}^{n-1} \la^{j-n}|\mu_n(t)-\xi_n(t)| \leq C(\de'), \nonumber
\end{align}
where we used the hyperbolicity of the hyperbolic set $\La_t$.
By Lemma \ref{prod-dist}, (\ref{close-to-1}) holds of $\de'$ is small enough.
Secondly, we show that
\[
\biggl| \frac{(f_t^n)'(\mu_0(t))}{(f_s^n)'(\mu_0(s))} - 1\biggr|\leq \vep_2,
\]
where $\vep_2=\vep_2(\de')$ is close to $0$. Put
$\la_{t,j}=f_t'(\mu_j(t))$. Since $\la_{t,j}$ are all analytic in $t$ we have
$$\la_{t,j}=\la_{0,j}(1+\sum_{|\al| \geq k_j} c_{\al,j} t^{\al}),$$ where $t = t_1 \cdot \ldots
\cdot t_{2d-2}$, $\al$ is a multi-index and $k_j \geq 1$. Moreover, the condition $R_a^j(v_a) \in \NN$ implies that
$n \leq -C \log|x(t)| \leq  -C' \log \|t\|$ for some constant $C'$,
where $\|t\|$ is the norm of $t$ viewed as a vector in $\C^{2d-2}$. We have
\[
\frac{(f_t^n)'(\mu_0(t))}{(f_s^n)'(\mu_0(s))}  = \prod \frac{\la_{t,j}}{\la_{s,j}} = \prod_{j=0}^{n-1}
\frac{(1+\sum_{|\al|=k_j} c_{\al,j} t^{\al} +
  \ldots)}{(1+\sum_{|\al|=k_j}c_{\al,j} s^{\al} + \ldots)}.
\]
Both the last numerator and denominator in the above equation can be
estimated by $1 + C'' n \|t\|^l$ and $1+C'' n \|s\|^l$ respectively, for
some constant $C''$ and integer $l \geq 1$. Since $n \leq -C' \log \|t\|$ the numerator and
denominator are bounded by $1 + \OO((\log \|t\|) \|t\|^l)$ and $1 + \OO((\log \|s\|)\|s\|^l)$ respectively, which both can be made
arbitrarily close to $1$ if $r > 0$ is small enough. From this the lemma follows.
\end{proof}

We reformulate Lemma 3.3 in \cite{MA3} in the following vector
form.
\begin{Lem} \label{first-dist}
Let $\vep > 0$. If $\de' > 0$ is sufficiently small, then for every $0 < \de'' < \de'$ there exists $r > 0$ such that the following holds. Let $a \in W$ and assume that $\xi_k(a) \in \NN$, for all $k \leq n$ and $|\xi_n(a)-\mu_n(a)| \geq \de''$. Then
\[
\| \xi_n'(a) - (f_a^n)'(\mu_0(a))x'(a) \|  \leq \vep \|\xi_n'(a)\|.
\]
\end{Lem}
\begin{proof}
First we note that by Lemma \ref{main-dist} we have
\[
\xi_n(a) = x(a) (f_a^n)'(\mu_0(a)) + \mu_n(a) + E_n(a),
\]
where, for instance $|E_n(a)| \leq |\xi_n(a) - \mu_n(a)|/1000$
independently of $n$ and $a$ if $\de'$ is small enough. Put $f_a'(\mu_j(a)) = \la_{a,j}$. Differentiating with respect to $a$ we get
\begin{equation} \label{xip}
\xi_n'(a) = \prod_{j=0}^{n-1} \la_{a,j} \biggr( x'(a) + x(a) \sum_{j=0}^{n-1} \frac{\la_{a,j}'}{\la_{a,j}} + \frac{\mu_n'(a) + E_n'(a)}{\prod_{j=0}^{n-1} \la_{a,j}}  \biggl).
\end{equation}

We claim that only the $x'(a)$ is dominant in (\ref{xip}) if $n$ is large so that $\de'' \leq |\xi_n(a)-\mu_n(a)| \leq \de'$. This means that, by Lemma \ref{main-dist},
\[
(1-\vep_1)\de'' \leq |x(a)| \prod_{j=0}^{n-1} |\la_{a,j}| \leq (1+\vep_1) \de' < 1,
\]
where $\vep_1 > 0$ is arbitrarily small provided $r > 0$ is small enough.
Since $\prod_{j=0}^{n-1} |\la_{a,j}| \geq \la^n$, for some $\la > 1$, taking logarithms and rearranging we get
\begin{equation}
(1-\vep_1) \sum_{j=0}^{n-1} \log |\la_{a,j}| \leq -\log |x(a)| \leq (1+\vep_1) \sum_{j=0}^{n-1} \log |\la_{a,j}|, \label{nsize}
\end{equation}
if $|\log \de''|  \ll |\log |x(a)||$, which is true if the perturbation $r > 0$ is chosen sufficiently small compared to $\de''$.
Since $|\la_{a,j}| \geq \la > 1$, this means that
\[
|x(a)| \sum_{j=0}^{n-1} \frac{\|\la_{a,j}'\|}{|\la_{a,j}|} \leq |x(a)| n C \leq -C |x(a)| \log|x(a)|.
\]
Finally $-|x(a)| \log|x(a)| / \|x'(a)\| \raw 0$ as $a \raw 0$ inside
$W$.

Now, $|E_n(a)|$ is uniformly bounded in $W$. Therefore,
$\|E_n'(a)\|$ is also uniformly bounded on compact subsets of
$W$ by Cauchy's Formula. By diminishing $r > 0$ slightly we can assume that
both $|E_n(a)|$ and $\|E_n'(a)\|$ are uniformly bounded on $W$. Hence, the last two
terms in (\ref{xip}) tend to zero as $n \raw \infty$, since also
$|\mu_n'(a)|$ is uniformly bounded. We have proved that
\[
\biggl\| \xi_n'(a) - x'(a) \prod_{j=0}^{n-1} \la_{a,j} \biggr\| \leq \vep \|\xi_n'(a)\|,
\]
if $|\xi_n(a)-\mu_n(a)| \leq \de'$ and $n \geq N$ for some $N$. Choose the perturbation $r$ sufficiently small so that this $N$ is at most the number $n$ in (\ref{nsize}). Since $\la_{a,j} = f_a'(\mu_j(a))$, the proof is finished.
\end{proof}

From this we deduce the following important Proposition (see also Proposition
3.4 in \cite{MA3}):
\begin{Prop} \label{comp-prop}
Let $\vep > 0$. If $\de' > 0$ is sufficiently small, then for every $0 < \de'' < \de'$, there exists $r > 0$ such that the following holds. Let $a \in W$ and assume that $\xi_k(a) \in \NN$, for all $k \leq n$ and $|\xi_n(a) - \mu_n(a)| \geq \de''$. Then
\begin{equation}
\| \xi_n'(a) - (f_a^n)'(v(a)) x'(a) \| \leq \vep \| \xi_n'(a)\|.
\end{equation}
\end{Prop}

More generally, we have a higher-dimensional form of Proposition 4.3 in \cite{MA}. Put
\[
Q_n = Q_{n,l}(a) = \xi_n'(a) / (f_a^n)'(v_l(a),a).
\]
\begin{Prop}
For every $\de > 0$ and sufficiently small $\de'' > 0$ there is an $r
> 0$ such that the following holds. Assume that the parameter $a \neq
0$, $a \in W$, satisfies $dist(\xi_n(a), Crit(f_a)) \geq \de$ for all $n \leq m$. Moreover, assume that $\de '' \leq |\xi_N(a)-\mu_N(a)| \leq \de'$ and
\begin{align}
|(f^n)'(\xi_N(a),a)| &\geq Ce^{\ga n}, \quad \textrm{for}
\quad n=0,\ldots,m,  \\
|\partial_a f(\xi_n(a),a)| &\leq B, \quad \forall n > 0,
\end{align}
where $\ga \geq (3/4) \log \la$ (see Subsection \ref{const} for definition of $\la$). Then we have, for $n=N,\ldots m$,
\begin{equation}
\|Q_n(a) - Q_N(a) \| \leq \|Q_N(a)\|/1000. \label{d}
\end{equation}
\label{da/dz}
\end{Prop}

\begin{proof}
First, we prove by induction, that
\begin{equation}
\|\xi_{N+k}'(a)\| \geq e^{\ga'(N+k)}, \label{expgrowth}
\end{equation}
where $\ga' = \min (\log \la /(2K), \ga/2)$ and $K$ is the order of $x(a)$ in $W$. From (\ref{x-series}) it is straightforward to
show that $$\|x'(a) \| \geq C |x(a)|^{\frac{K-1}{K}},$$ for
some constant $C$. Then we get
\begin{equation}
\|\xi_N'(a)\| \geq (1/2)|(f_a^N)'(v(a))|\|x'(a)\| \geq (C/2) |(f_a^N)'(v(a))| |x(a)|^{\frac{K-1}{K}}.
\end{equation}
By Lemma \ref{main-dist} and since $\de '' \leq |\xi_N(a)-\mu_N(a)| \leq \de'$, we get
\[
|(f_a^N)'(v(a))| |x(a)| \geq \de''/2.
\]
If the perturbation $r > 0$ is chosen sufficiently small, so that $N$ becomes sufficiently large, we deduce that
\[
\|\xi_N'(a)\| \geq (C/2) |(f_a^N)'(v(a))| |x(a)|^{\frac{K-1}{K}} \geq (C \de'' /4) |(f_a^N)'(v(a))|^{1/K} \geq e^{\ga' N},
\]
where $\ga' \geq \min (\log \la /(2K), \ga/2)$.

So, assume that
\[
\|\xi_{N+j}'(a)\| \geq e^{\ga'(N+j)},  \text{ for all $j\leq k$}.
\]
We want to prove that
\[
\|\xi_{N+j}'(a)\| \geq e^{\ga'(N+j)},  \text{ for all $j\leq k+1$}.
\]
First note that the assumption $dist(\xi_n(a),Crit(f_a)) \geq \de$,
with $\de=e^{-\De}$, implies
\begin{equation}
|f'(\xi_j(a),a)| \geq C_1^{-1}e^{-\De K},
\label{bbaa}
\end{equation}
for some $C_1 > 0$. By the Chain Rule we have the recursions
\begin{align}
\frac{\partial f^{n+1}(v(a),a)}{\partial z} &= \frac{\partial
  f(\xi_n(a),a)}{\partial z} \frac{\partial f^n(v(a),a)}{\partial z}
  \label{dz}, \\
\frac{\partial f^{n+1}(v(a),a)}{\partial a} &= \frac{\partial
  f(\xi_n(a),a)}{\partial z} \frac{\partial f^n(v(a),a)}{\partial a} +
\frac{\partial f(\xi_n(a),a)}{\partial a}. \label{da}
\end{align}
Now, the recursion formulas (\ref{dz}) and (\ref{da}), together
with (\ref{bbaa}) gives
\begin{align}
\|\xi_{N+k+1}'(a)\| &\geq
 |f_a'(\xi_{N+k}(a))|\|\xi_{N+k}'(a)\|\biggl(1-
\frac{\| \partial_a f_a(\xi_{N+k}(a)) \| }
{|f_a'(\xi_{N+k}(a))| \|\xi_{N+k}'(a)\|} \biggr) \nonumber \\
&\geq
|(f_a^{k+1})'(\xi_N(a))|\|\xi_N'(a)\| \prod_{j=0}^k \biggl(1-\frac{\|\partial_a
   f_a(\xi_{N+j}(a)) \|}{|f_a'(\xi_{N+j}(a))|\|\xi_{N+j}'(a)\|} \biggr)
\nonumber \\
&\geq
e^{\ga (k+1)} e^{\ga' N} \prod_{j=0}^k (1-B' e^{K\De}e^{-\ga'(N+j)} ) . \nonumber
\end{align}
We have
\[
\sum_{j=0}^{\infty} B'e^{\De-\ga'(N+j)} < \infty,
\]
and the sum can be made arbitrarily small if $N$ is large enough. Therefore,
\begin{equation}
\|\xi_{N+k+1}'(a)\| \geq e^{(\ga-\ga')(k+1)} e^{\ga' (N + k + 1)}
\prod_{j=0}^k (1-B' e^{K \De-\ga'(N+j)} ) \geq e^{\ga' (N+k+1)},  \nonumber
\end{equation}
if $N$ is large enough, since $\ga' \geq 2K\al$, (here $B' = BC_1$). Hence (\ref{expgrowth}) follows.

To continue the proof, first note that
\begin{align}
\| \xi_{N+k}'(a) - \xi_N'(a) &\prod_{j=N}^{N+k} f_a'(\xi_j(a)) \| =
\| \sum_{j=N}^{N+k} \partial_a f_a(\xi_j(a)) \prod_{i=j}^{N+k} f_a'(\xi_{i+1}(a)) \| \nonumber \\
&\leq
\sum_{j=N}^{N+k} \| \partial_a f_a(\xi_j(a)) \| \prod_{i=j}^{N+k} | f_a'(\xi_{i+1}(a))|. \label{summ}
\end{align}
Let us put $\la_n = f_a'(\xi_n)$, $\mu_n = \partial_a f_a (\xi_n(a))$ and $\xi_n'(a) = \xi_n'$.
The first term in (\ref{summ}) is
\[
\| \mu_N \| \prod_{i=N+1}^{N+k-1} |\la_i| \leq \frac{\| \mu_N \|}{|\la_N| \| \xi_N'\|} \prod_{i=N}^{N+k-1} |\la_i| \| \xi_N'\| \leq B' e^{\De K} e^{-\ga'N} \prod_{i=N}^{N+k-1} |\la_i| \| \xi_N' \|.
\]
Put $C=B'e^{\De K}$. The $nth$ term in (\ref{summ}) becomes
\begin{align}
\| \mu_{N+n} \| &\prod_{i=N+n}^{N+k-1} |\la_i| = \frac{\| \mu_{N+n}  \|}{|\la_{N+n-1}| \| \xi_{N+n-1}' \|} \prod_{i=N+n-1}^{N+k-1}|\la_i| \| \xi_{N+n-1}' \| \nonumber \\
&\leq
Ce^{-\ga' (N+n-1)} \prod_{i=N+n-1}^{N+k-1} |\la_i|\biggl( \prod_{j=N}^{N+n-2} |\la_i| \| \xi_N' \| + \sum_{j=N}^{N+n-2} \| \mu_j \| \prod_{i=j+1}^{N+n-2} |\la_i| \biggr) \nonumber \\
&\leq
Ce^{-\ga' (N+n-1)} \prod_{i=N}^{N+k-1} |\la_i| \| \xi_N \| + Ce^{-\ga' (N+n-1)} \sum_{j=N}^{N+n-2}
\| \mu_j \| \prod_{l=j+1}^{N+k-2} |\la_l|. \nonumber
 \end{align}
The last sum is the sum of the first $n-1$ terms in (\ref{summ}). As induction assumption the first $n-1$ terms in (\ref{summ}) is less than $1$. This means that the $n$th term is at most $Cne^{-\ga' (N+n-1)}$, which is again of course less than $1$ if $N$ is chosen sufficiently big. We get finally
\[
\| \xi_{N+k}'(a) - \xi_N'(a) \prod_{j=N}^{N+k} f_a'(\xi_j(a)) \| \leq \sum_{n=N}^{\infty} Cne^{-\ga' (N+n-1)} \| \xi_N'(a) \| \prod_{j=N}^{N+k} |f_a'(\xi_j(a))|.
\]
So, if $N$ is big enough,
\begin{equation}
\|Q_{N+n}(a)-Q_N(a)\| \leq  \|Q_N(a)\|/1000. \nonumber
\end{equation}
\end{proof}

From Proposition \ref{da/dz} and \ref{comp-prop} we see that the space
derivative and parameter derivative are comparable up to a
multiplicative quantity, namely $x'(a)=x_j'(a)$ for some $j$. However,
since $x'(a)$ generally is not constant, we want to restrict the
parameters such that $x'(a)$ does not vary much. To this end it is
naturally to restrict to sets where (\ref{xprim}) holds.

If $x'(t)$ is not
constant, then by Lemma \ref{circles} the set
of parameters which satisfies the condition (\ref{xprim}) contains a
$k$-Whitney disk $D_0
\subset W$ for some $0 < k < 1$ only depending on the function $x$.

We need to know that a Whitney (parametric) disk grows to the large
scale before it leaves $\NN$. This is the content of the following
lemma. Let the angle between vectors $x$ and $y$ be $\arccos (x \cdot y /(\|x\|
\|y\|)) \in [0,\pi ]$. Given two hyper planes $H_1$ and $H_2$ we say that the angle $\theta \in [0,\pi]$
between them is defined by
\[
\cos \th = \inf\limits_{\bar{x} \in H_1} \sup\limits_{\bar{y} \in H_2} \frac{ \bar{x} \cdot \bar{y} }
{\| \bar{x} \| \| \bar{y} \|}.
\]
Then we can also talk about angles between Whitney disks and hyper planes since every Whitney disk
is contained in some unique hyper plane (with dimension equal to the
dikension of the Whitney disk). Moreover, if $D_0$ is a Whitney disk and $H$ a hyperplane, then
the diameter of the orthogonal projection of $D_0$ onto $H$ is $diam(D_0) \cos(\th)$, where $\th$ is the angle
between $D_0$ and $H$.
\begin{Lem} \label{initdist}
If $r > 0$ is sufficiently small, there exists a number $0 <k <1$ only
depending on the function $x$ such that the following
holds. Let $D_0=\B(a_0,r_0) \subset W$ be a $k$-Whitney disk ($D_0$ has
dimension $\leq 2d-2$) such that the angle $\th$ between $D_0$ and $x'(a_0)$ is $\th \neq \pi/2$:
There is an $n > 0$ and a number
$S=S(\de',\th)$, such that the set $\xi_{n}(D_0) \subset \NN$ and has diameter at
least $S$. Moreover, we have low argument distortion, i.e.
\begin{equation}
\| \xi_k'(a) - \xi_k'(b) \| \leq \| \xi_k'(a) \|/100,
\label{xinn}
\end{equation}
for all $a,b \in D_0$ and all $k \leq n$.
\end{Lem}
\begin{proof}
Choose $n$ maximal such that $\xi_k(a_0) \in \NN$ for all $k \leq n$ and
\[
(\de'+\de'')/(2M_0) \leq |\xi_n(a_0)-\mu_n(a_0)| \leq (\de'+\de'')/2,
\]
where $M_0$ is the supremum of $|f_a'(z)|$ over all $a \in \B(0,r)$ and $z \in \hat{\C}$.

Proposition \ref{comp-prop} holds for all $a \in W$ satisfying
\begin{equation} \label{debis}
\de'' \leq |\xi_n(a)-\mu_n(a)| \leq \de'.
\end{equation}
Since $x'(a)$ has bounded distortion on Whitney disks by Lemma
\ref{circles}, for parameters $a,b \in D_0$ satisfying (\ref{debis}) we have good control of the geometry:
\begin{equation} \label{xiprim}
\| \xi_n'(a) - \xi_n'(b) \| \leq \|\xi_n'(a)\| /100.
\end{equation}
Hence, $\xi_n$ is almost linear in $D_0=\B(a_0,r_0)$ if (\ref{debis}) is satisfied.

Assuming that (\ref{debis}) holds for all $a \in D_0$, we want to estimate the diameter $d$ of the
set $\xi_n(D_0)$. It is, up to very low distortion, precisely the orthogonal projection
of $\xi_n'(a_0)$ onto the hyperplane where $D_0$ lies.
The diameter $d$ can be estimated by
\begin{equation}
d \geq
\| \xi_n'(a_0) \| | diam(D_0) \cos (\th) | \geq (1/2)
|(f_{a_0}^n)'(\mu_0(a_0))| \|x'(a_0)\| \| k_0a_0 \| \cos (\th),
\nonumber
\end{equation}
where we also used Proposition \ref{comp-prop}.
If the Whitney disk $D_0$ is too large such that (\ref{debis}) is not
fullfilled, then $\xi_n(D_0)$ fails to be a subset of
\[
A(\de'',\de',\mu_n(a_0)) = \{ z: \de'' \leq |z-\mu_n(a_0)| \leq
\de' \}
\]
and we may have to diminish $r_0$. However, with $r_0=k_0 \| a_0 \|$
we can choose $\de'' > 0$ sufficiently small so that at least if $k_0
\leq 1/2$, then $\xi_n(D_0) \subset A(\de'',\de',\mu_n(a_0))$.

By Lemma \ref{main-dist},
\begin{align}
(\de'+\de'')/(2M_0) &\leq |\xi_n(a_0)-\mu_n(a_0)| \leq 2
|(f_{a_0}^n)'(\mu_0(a_0))||x(a_0)|. \nonumber
\end{align}
Thus,
\begin{equation}
\frac{d}{\de'} \geq C
\frac{|(f_{a_0}^n)'(\mu_0(a_0))|}{|(f_{a_0}^n)'(\mu_0(a_0))|}
\frac{\|k_0 a_0\| \| x'(a_0) \|}{|x(a_0)|}.\nonumber
\end{equation}
The number $\|k_0 a_0\| \|x'(a_0) \|
/ |x(a_0)|$ is bounded from below in $W$. Hence
there is some constant $C'$ so that $d/\de' \geq C'$.
So the diameter $d$ of the set $\xi_n(D_0)$ is greater than some $S =
S(\de', \th)$.
Also, by (\ref{xiprim}), we have bounded argument distortion for all $a,b \in D_0$.
\end{proof}

Hence, a Whitney disk $D_0 \subset W$ will grow to size $S$ under
the map $\xi_n$ before $\xi_n(D_0)$ leaves $\NN$. At the same time we
have strong control over the distortion up to the scale $S$. Let us formalize and say that we have {\em strong distortion estimates in $D_0$ up to time $n$} if
\begin{equation}
\| \xi_k'(a) - \xi_k(b) \| \leq \| \xi_k'(a) \| /100,
\label{xinn2}
\end{equation}
holds for all $a,b \in D_0$ and for all $k \leq n$. If it is clear
what $n$ is, we just say strong distortion estimates in $D_0$.

\comm{
Similarly, if we start with a whole $1$-dimensional disk $B(0,r)
\subset W$ in which the corresponding $x(a)$ (for some critical point) has finite order contact, it also grows to the large scale before leaving $\NN$. The following follows from Lemma 3.5 in \cite{MA3}.

\begin{Lem} \label{X}
There exist numbers $R > 0$, $S > 0$ only depending on $f=f_0$ such that for any $0 < r \leq R$ the
following holds. Assume that $x(a)$ has finite order contact in $B(0,r)$. Let $n$ be maximal for which $diam (\xi_n(B(0,r)))
\leq S$. Then $\xi_n(B(0,r)) \subset \NN$ and for every Whitney disk $D_0 \subset B(0,r)$, we have
strong distortion estimates.

Moreover, there is some $\vep > 0$ depending on $\de'$, such that
$\xi_n(B(0,r))$ is almost round, meaning that there are two balls
$B_1$ and $B_2$ such that $B_1 \subset \xi_n(B(0,r)) \subset B_2$ and
such that $\mod(B_2 \sm B_1) < \vep$. The number $\vep > 0$ can be
chosen arbitrarily small if $\de'$ is small enough.
\end{Lem}
}

Finally we will use the following distortion lemma for the so called
free period, i.e. when $\xi_n(E)$ has left $\NN$, for some set
$E$. The following follows directly.
\begin{Lem}[Extended Distortion Lemma]\label{extend-dist}
Let $N \in \N$. For any $\vep > 0$ and neighbourhood $U$ of $Crit(f_0)$, there exists an $r > 0$ and $S' > 0$ such that the following holds. Let $a,b \in \B(0,r)$ and assume that $z,w \in \NN$ are such that $f^k(z,a), f^k(w,b) \notin U$ and $|f^k(z,a)-f^k(w,b)| \leq S'$ for all $k=0, \ldots,n$, where $n \leq N$.
Then
\[
\biggl| \frac{(f^n)'(z,a)}{(f^n)'(w,b)} -1 \biggr| < \vep.
\]
\end{Lem}

The bound $N$ will come from the following lemma.

\begin{Lem} \label{return}
There exists an $r > 0$ such that the following holds. Fix $d > 0$ and let $\SS$ be a family of disks with diameter $d$ which cover the Julia set $J(f)$ of the starting function $f_0$ and such that each disk $S \in \SS$ is centered at a point in $J(f)$. Then there exists some constant $N$ such that
\begin{equation}
\inf \{ m : f_0^m(S) \supset \oli{U} \} \leq N, \label{cover}
\end{equation}
for every disk $S \in \SS$.
\end{Lem}
\begin{proof}
Since $f^n$ is not normal on the Julia set, for each point $z \in
J(f)$ there is some (smallest) $N(z)$ for which $f^{N(z)}(S) \supset
\oli{U}$. For each $z$ there is some neighbourhood for which $N(z)$ is
constant. Since $J(f)$ is compact there is a constant $N$ such that
\[
\inf \{ m : f^m(S) \supset \oli{U} \} \leq N,
\]
for any $S \in \SS$. The lemma follows.
\end{proof}
Arrange the disks in the family $\SS$ so that any disk $D$ of diameter $d$ for which there exists a point $z \in D \cap J(f) \neq \emptyset$ such that $dist(z,\partial D) \geq d/4$, there exists some $S \in \SS$ such that $S \subset D$.

\section{Closing the critical orbits} \label{Closing}

Although we have shown that we have strong distortion estimates on
small Whitney disks, we start with a full-dimensional disk
$\B_0 \subset W$. By Lemma \ref{initdist},
$\xi_n(\B_0)$ grows to some large scale
size $S=S_0$ under strong distortion estimates, for some $n > 0$.
Assume that $x_1(a)$ has finite order contact and assume that we have found a solution $\xi_{n+m,1}(a_0)=c_{1}(a_0)$ for some $a_0 \in \B_0$.
Let $\B_1'$ be the connected component of the set $\{a \in
W: \xi_{n+m,1}(a)=c_{1}(a) \}$ containing $a_0$. In order to get good geometry control of this manifold, we need to restrict to a set $\B_1 = D\cap \B_1'$, where $D$ is a Whitney disk.

A proof of the following general result can be found in \cite{JM} p. 11, for instance.
\begin{Lem} \label{Milnor}
Given an analytic function $F$ from $\C^n$ to $\C$, where
$F(z_0)=w_0$. Then a relatively open subset $E \subset F^{-1}(w_0)$ is a submanifold if for all $z \in E$ we have $F'(z) \neq 0$.
\end{Lem}

Hence, set of parameters $a \in W$ satisfying
$F_{1}(a)=\xi_{n+m,1}(a)-c_{1}(a)=0$ is a submanifold, apart from
a set of singularities. In the next lemma we deal the problem of singularities.
\begin{Lem} \label{nosing}
Assume that $A \subset W$ is a connected manifold, such that $\xi_{k,l}(A) \in \NN$, for all $k \leq n$. Assume that $F_l(a) =
\xi_{m+n,l}(a)-c_{l}(a)=0$ for some $a \in A$ and that
$\xi_{k,l}(a) \cap U_l = \emptyset$, for all $k \leq n+m-1$, where
$U_l$ is a $\de_l$-neighbourhood of $Crit(f)$, $\de_l \leq
\de$. Assume moreover that $U_l$ has the property that the first
return time into itself is at least $2m$ and that every $c_j(a)$
belongs to a $\de_l^{10}$-neighbourhood of $Crit(f)$, for $a \in
\B(0,r)$.

Then if $ r > 0$ is sufficiently small, then $\|\xi_{n+m,l}'(a)\| > 100
\|c_l'(a)\|$ for all zeros of $F_l$ inside $A$. In particular,
there are no singularities of $F_l$ on its set of zeros inside
$A$.
\end{Lem}
\begin{proof}
The condition on $U_l$ means that any solution to $F_l(a)=0$ for $a
\in A$ must have that $\xi_{k,l}(a) \cap U_l = \emptyset$ for all $k \leq n+m-1$.

We have $|(f_a^n)'(v_a)| \geq e^{\ga n}$, for some $\ga \geq 2\uli{\ga}$, by the definition of $\NN$.
We can choose $r$ so that $m/n$ is
arbitrarily small, i.e. during the iterates $n+1,\ldots,n+m$ we do not
lose much in derivative. In other words,
\[
|(f_a^{n+m})'(v(a))| \geq  e^{\ga n} |(f_a^m)'(f_a^n(v(a)))| \geq e^{\ga_1 n},
\]
where $0 < \ga_1 < \ga$. Indeed, we can get $\ga_1$ as close to $\ga$
as we want. Choose $\ga_1$ so that $\ga_1  \geq \uli{\ga}$. By Proposition \ref{da/dz},
\[
\|\xi_{n+m,l}'(a)\| \geq e^{\ga_1 (n+m)} \|x'(a)\| \geq  e^{\ga' (n+m)},
\]
for some $\ga' \geq (1/k)\ga_1$ (see proof of Proposition \ref{da/dz}). Choosing $n$ sufficiently large
(i.e. $r > 0$ sufficiently small), we can therefore ensure that
\[
\|\xi_{n+m,l}'(a)\| > 100 \|c_l'(a)\|,
\]
for all $a \in V  \cap A$, where $V$ is a neighbourhood the solution set
$\xi_{n+m,l}(a)-c_{l}(a)=0$. Hence $F'(a) \neq 0$ for all $a \in V \cap A$.
\end{proof}

Passing on to a certain subset of $\B_1' \subset W$, we want
to show that this set has low curvature viewed as a surface
embedded in $W$. To see what conditions are imposed on such a set, we begin with showing that $\xi_{n+m,1}(a)$ has bounded distortion in $\B_1'$ if $|\xi_{k,1}(a)-\xi_{k,1}(b)| \leq T$ for all $k \leq n+m$ all $a,b \in \B_1'$ and some number $T > 0$, depending on $U=U$.

\begin{Lem} \label{directdist}
Let $U' \supset U$ be a $10\de$-neighbourhood of
$Crit(f) \cap J(f)$ and let $\vep  > 0$. Then there exist $r > 0$, $T
> 0$ where $T$ only depends on $U$ and $\vep$, such that
the following holds. Assume that $a_0 \in W \subset \B(0,r)$,
$\xi_{n,i}(a_0) = c_{j}(a_0)$ and $\xi_{k,i}(a_0) \cap U' =
\emptyset$ for all $k \leq n-1$. Then we have
\begin{equation}
\| \xi_{n,i}'(a_0) - \xi_{n,i}'(a) \| \leq \vep \| \xi_{n,i}'(a)    \|,
\label{distor}
\end{equation}
if $|\xi_{k,i}(a_0)-\xi_{k,i}(a)|\leq T$ for all $k \leq n$. Moreover, for
each such $a$, we have that $\xi_{k,i}(a) \cap U =
\emptyset$, for all $k \leq n-1$.
\end{Lem}
\begin{proof}
Write $\xi_{k,i}=\xi_k$.
Let $S > 0$ and $0 < k < 1$ be from Lemma \ref{initdist}. Assume that $n_1$ is maximal such that
\[
diam(\xi_{n_1}(D_0)) \leq S, \quad \text{ and } \quad \xi_{n_1}(D_0) \subset \NN,
\]
where $D_0 \subset W$ is a full-dimensional $k$-Whitney disk with center at $a_0$.
Lemma \ref{initdist} implies that (\ref{distor}) holds if $n$ is replaced by $n_1$ for $T = S$ and when $a \in D_0$. We will show that (\ref{distor}) holds after $n$ iterates for some $T \leq S$.

Let $S' > 0$ be the constant in Lemma \ref{extend-dist}, given by $N=n-n_1$, $U$ and some suitable sufficiently small $\vep > 0$. Choose $T \leq S'$ maximal such that the condition $|\xi_k(a_0)-\xi_k(a)|\leq T$ for all $k \leq n$, implies that
$\xi_{k}(a) \cap U = \emptyset$ for all $k \leq n-1$. Note that $T$ only depends on $U$. Then Lemma \ref{extend-dist} together with Lemma \ref{main-dist} implies that
\[
\biggl| \frac{(f_{a_0}^{n})'(v_{a_0})}{(f_a^{n})'(v_a)} - 1 \biggr| \leq \vep_1,
\]
where $\vep_1$ is some suitable sufficiently
small positive number ($\vep_1$ depends on $\vep$). Moreover, by a the same argument as in the beginning of the proof of Proposition \ref{da/dz}, we have $|(f_x^k)'(v_x)| \geq e^{\ga' k}$, for $x=a,a_0$ and all $k \leq n$ and some $\ga' \geq \uli{\ga}$.
Proposition \ref{da/dz} and Lemma \ref{circles} imply that (\ref{distor}) holds since $a \in D_0$.
The lemma is proved.
\end{proof}

Let us assume that $\xi_{n+m,l}(a)-c_l(a)=0$ and that $\xi_{n+m,l}(a)$ satisfies
the assumptions in the above lemma. Put $F(a)=\xi_{n+m,l}(a)-c_l(a)$.
Lemma \ref{directdist} implies immediately that $\| \nabla F(b) - \nabla F(a) \|$ is small if $b$
satisfies $|\xi_k(a)-\xi_k(b)| \leq T$ for all $k \leq m+n-1$. Hence if that holds for all parameters $a$ in some submanifold in $W$, it means this manifold has low curvature.


\begin{Def}
Suppose that $E$ is an open $n$-dimensional connected manifold parameterised by some open set $D \subset \C^n$, where $\phi: D \raw E$ is a
diffeomorphism and $E=\phi(D)$ and $\phi(\partial D) = \oli{E} \sm
E$.
We say that $E$ is {\em almost planar} if
\[
\| \phi'(x) - \phi'(y) \| \leq 1/100,
\]
for all $x,y \in D$.

If, in addition, $D$ is a disk and
\[
diam(E) \geq k dist(E,0),
\]
then we say that $E$ is an almost
planar $k$-Whitney disk. Moreover, if $d=diam(E)$, then we say that the
radius of $E$ is $r=d/2$.
For any $0< r' < r$ by $ddist(x,\partial E) \geq r'$ we mean the set $\{ x \in E:
dist(x,\oli{E} \sm E) \geq r'\}.$
\end{Def}

Let us now assume that we are in the $l$th step so that we have
constructed a nested sequence of almost planar disks $\B_{k+1}
\subset \B_k$, $0 \leq k \leq l-1$,
such that each $\B_k$ has that $F=F_k(a)= \xi_{n_k+m_k,k}(a)-c_k(a)=0$
for all $a \in \B_k$. In fact, since the sequence is nested,
$F_k(a)=0$ holds for all $1 \leq k \leq l$ in $\B_l$.
Moreover, note that the normal vectors to the solution sets are
$F_k'(a) = \xi_{n_k+m_k,k}'(a) - c_k'(a)$. These vectors are
arbitrarily close to $\xi_{n_k+m_k,k}'(a)$ because
$\| \xi_{n_k+m_k,k}'(a) \|$ is much larger than $\| c_k'(a) \|$,
which follows from Lemma \ref{nosing}. Also, $\xi_{n_k+m_k,k}'(a)$ is well approximated by $
(f_a^{n_k+m_k})'(v_k(a)) x_k'(a)$ by Proposition \ref{comp-prop}. Hence we have the following important fact:
\begin{Fact}
The normal vectors to the solution sets $F_k(a)=0$ are, up to arbitrary low distortion, parallell to $x_k'(a)$.
\end{Fact}

A priori, Lemma \ref{circles} only applies to Whitney disks rather than almost planar Whitney disks.
However, since $\B_l$ is almost planar, it is uniformly well approximated by a
hyperplane, i.e. the tangent space at points on $\B_l$ does not vary much on $\B_l$. By Lemma \ref{initdist} it then follows that
if $c_{l+1}$ has finite order contact, and if $x_{l+1}'(a)$ is not perpendicular to $\B_l$ for $a \in \B_l$,
the set $x_{l+1}(\B_l)$ will then grow to the large scale $S=S_l$ before
leaving $\NN$, i.e. $\xi_{n,l+1}(\B_l) \subset \NN$ contains a
disk of diameter $S_l$. By Lemma \ref{return} there is
some $N=N_l$ depending on $S_l$ such that $\xi_{n+m,l+1}(\B_l)$ covers
$\oli{U}$ for some $m \leq N_l$. Clearly, if $\B_l$ is an almost
planar $k_l$-Whitney
disk, then $N_l$ depends only on $k_l$.

We now prove that if a given solution is found to $F_l(a)=0$, then the parameters satisfying the conditions in Lemma \ref{directdist} will contain a $k_{l+1}$-Whitney disk. Recall that the sets $U_l$ are $\de_l$-neighbourhoods around the critical points on the Julia set for $f$. Let $U_l' \supset U_l$ be
$10\de_l$-neighbourhoods around these critical points. Let $M_0=\max
|f_a'(z)|$ where the maximum is taken over all $(z,a) \in \hat{\C}
\times \oli{\B}(0,r)$.
\begin{Lem}[Inductive Lemma I] \label{induI}
Assume that $\B_{l} \subset \B_0 \subset W$ is an almost planar Whitney
disk of diameter $2r_l  \geq k_l dist(\B_l,0)$
and for which every $a \in \B_l$ has
that $\xi_{n_k+m_k,k}(a)-c_{k}(a)=0$ for all $1 \leq k \leq
l$ (if $l=0$ we have no solutions so far). Assume that we have found a solution to
$$\xi_{n_{l+1}+m_{l+1},l+1}(a_0)-c_{l+1}(a_0)=0$$ for some $a_0
\in \B_l$, such that $ddist(a_0,\partial \B_l) \geq r_l/2$ and such that $\xi_{n_{l+1}}(\B_l) \subset \NN$
and $m_{l+1} \leq N_l$, where $N_l$ only depends on $k_l$.

Then if $r > 0$ is sufficiently small, and if $dim(\B_l) > 1$, there
exists an almost planar $k_{l+1}$-Whitney disk $\B_{l+1} \subset \B_l$ of codimension $1$ (in $\B_l$), where $k_{l+1}$ only depends on $k_l$. For every
$a \in \B_{l+1}$ we have
$\xi_{n_{l+1}+m_{l+1},l+1}(a)-c_{l+1}(a)=0$. If $dim(\B_l) = 1$, the set $\B_{l+1}$ might reduce to a single point.
\end{Lem}
\begin{proof}
We can without loss of generality assume that $n_{l+1}$ is the largest integer such that $\xi_{n_{l+1}}(\B_l)
\subset \NN$ and such that $\xi_{n_{l+1}}(\B_l)$ contains a disk of
diameter $S_l$, where $S_l > 0$ is the large scale from Lemma
\ref{initdist}. Now choose $U_{l}'$ such that
$\xi_{k,l+1}(a_0) \cap U_{l}' = \emptyset$ for all $k \leq m_{l+1}+n_{l+1}-1$ and that the first return
time from $U_{l}'$ to itself is at least $2N_l$. Hence $U_{l}$ (which is a
$\de_{l}$ neighbouhood of $Crit(f_0)$) depends only on $m_{l+1} \leq N_l$,
given that $r > 0$ is sufficiently small. The condition on $r> 0$ is that $c(a) \in U(c(0),\de_{l}^{10})$
for all critical points $c(a)$, $a \in \B(0,r)$.


Put $E= \{a \in W: \xi_{m+n,l+1}(a)-c_{l+1}(a)=0 \}$. By assumption, the set $E \cap \B_l$ is non empty.
Since $dim(\B_l) > 1$ and $E$ has codimension $1$, using Lemma \ref{nosing} with $A=\B_l$, we see that the set
$E \cap \B_l$ is a smooth manifold. Moreover, we must have $dim(E \cap \B_l) \geq 1$. If $dim(\B_l)=1$
then $E \cap \B_l$ might reduce to a single point.

Let $\xi_{k,l+1}=\xi_{k}$ and put $m_{l+1}=m$ and $n_{l+1}=n$.
According to Lemma \ref{directdist}, to have good geometry control of a manifold in $W$,
any parameter $b$ in this manifold must satisfy
\begin{equation}
| \xi_k(a_0) -\xi_k(b) | \leq T_{l}
\label{cond}
\end{equation}
for some $T_{l} > 0$ depending on $U_l$ for all $k \leq n+m$. We will show that the set of such parameters
$b$ satisfying (\ref{cond}) contains a $k_{l+1}$-Whitney disk $\B'$,
centered at $a_0$, where $k_{l+1}$ only depends on $k_{l}$.
Since $m \leq N_l$, the expansion $|(f_a^m)'(z)| \leq C_l=C_l(N_l)$ is
bounded and depends only on $N_l$. Hence,
$|\xi_k(a_0) -\xi_k(b)| \leq T_l$ for all $k \leq n+m$ if
$|\xi_{k}(a_0)-\xi_{k}(b)| \leq S_{l+1}'$ for all $k \leq n$, where
$S_{l+1}' \leq T_l/(2C_l)$. Now put
$S_{l+1}=S_{l+1}'/(2M_0)$ (then $S_{l+1}$ will be the new large scale).
We get that (\ref{cond}) holds for a $k_{l+1}$-Whitney $\B'$
disk centered at $a_0$, where $k_{l+1}$ is
minimal such that $\xi_n(\B')$ contains a
disk of diameter $S_{l+1}$ (then $diam(\xi_n(\B')) \leq S_{l+1}'$). Since $\xi_n$ is almost linear on
$k$-Whitney disks according to Lemma \ref{initdist}
(where also $k_{j+1} \leq k_j$, $k=k_0$), we get that $k_{l+1}/k_l  - S_{l+1}/S_l$ is arbitrarily close to zero
(hence $k_{l+1} \approx k_l S_{l+1}/S_l$).

Moreover, the set $E \cap \B'$ must be almost planar in $\B'$. It follows that $\B_{l+1} = \B' \cap (E \cap \B_l)$ is an almost
planar $k_{l+1}$-Whitney disk in $\B_l$.
Finally, we see that $k_{l+1}$ only depends on $S_{l+1}$, $k_l$ and
$S_l$. Clearly, $S_{l+1}$ depends only on $T_l$ and $C_l$.
Now $T_{l}$ depends on $U_l$ which in turn depends on $N_l$ and
moreover $C_l$ depends clearly on $N_l$.
Finally, $N_l$ depends only on $S_l$ (the previous large scale) which
in turn depends on $k_l$. Hence $k_{l+1}$ depends only on $k_l$.
The lemma is proved.
\end{proof}

\begin{Lem}[Inductive Lemma II] \label{induII}
Assume that we have found an almost planar $k_l$-Whitney disk $\B_l$ (of diameter $2r_l$) and a list of critical
points $C_l=\{c_{1},\ldots,c_{l} \}$ depending on the
parameter $a$ such that for each $c_{k} \in C_l$ we have $\xi_{n_k+m_k,k}(a)-c_{k}(a)=0$ for all $a
\in \B_l$ and all $1 \leq k \leq l$. Assume that $n_{l+1}$
is maximal such that $\xi_{n_{l+1},l+1}(\B_l) \subset \NN$.

Then if $r > 0$ is sufficiently small there exists a solution to
$$\xi_{n_{l+1}+m_{l+1},l+1}(a)-c_{l+1}(a)=0$$ for some $a \in \B_l$, such that
$ddist(a ,\partial \B_l) \geq r_l/2$, where $m_{l+1} \leq N_l$, and
$N_l$ is an integer which only depends on $k_l$.
\end{Lem}
\begin{proof}
Put $\xi_{n_{l+1},l+1}=\xi_n$.
It follows from Lemma \ref{initdist} that $\xi_n(\B_l)$ contains a disk of diameter at least
$S_l$ before leaving $\NN$ (where $S_l$ depends only on $k_l$).
Indeed, the disk $\B_l$ is an intersection of small manifolds
determined by $F_j(a)=\xi_{n_j,j}(a)-c_j(a)=0$. Each of these manifolds has
normal vectors equal to $F'(a) = \xi_{n_j,j}(a)-c_j'(a)$. By the Fact on p. 15,
we have that $F_j'(a)$ is almost parallell to $x_j'(a)$.
The vectors $x_j'(0)$ are all linearly independent by Theorem
\ref{ttrans}. This implies that if $\th$ is the angle between a tangent hyper-plane to
$\B_l$ and a tangent hyper-plane to $x_{l+1}'(a_0)$ (see before Lemma
\ref{initdist} for definition) then $cos(\th)$ is
bounded away from $0$ (since both these surfaces are almost planar their
tangent hyper-planes do not vary much).

Since $\xi_n(a) \in J(f_0)$ for all $a \in \B_l$
where $ddist(a,\partial \B_l) \geq (3/4)r_l$, there is some (maximal)
almost planar disk $\B_l' \subset \B_l$ centered at $a$ such that $ddist(b,\partial \B_l) \geq r_l/2$, for
all $b \in \B_l'$. The set $\xi_{n_{l+1}}(\B_l')$ will contain a disk of diameter $S_l/8$ centered at the
Julia set of $J(f_0)$. By Lemma \ref{return} there is an integer $N_l$, where
$N_l=N_l(S_l)$ only depends on $S_l$, such that $f_0^{m}(\xi_n(\B_l')) \supset \oli{U_l}$, $m \leq N_l$.
Let $m$ be defined by
\[
m = \inf \{ k > 0: f_0^k(\xi_n(\B_l')) \supset \oli{U_l}\}.
\]
Since the parameter dependence can be made arbitrarily
small under $N_l$ iterates, by choosing $r > 0$ sufficiently small, we
can also ensure that $f_a^m(\xi_n(\B_l')) = \xi_{n+m}(\B_l') \supset
c_{l+1}(\B_l')$. Hence there is a solution to $\xi_{n_{l+1}+m_{l+1},l+1}(a)-c_{l+1}(a)=0$
inside $\B_l'$. By the definition of $\B_l'$, we have $ddist(a,\partial \B_l) \geq r_l/2$.
\end{proof}

\begin{Rem}
The dependence of the constants $U_l,T_l,S_l,N_l$ and $r >0$ might seem intricate.
Let us clarify the feasibility of choosing these constants in a consistent way.
Put $S=S_0$ and $k=k_0$ in Lemma \ref{initdist}. The constants $N_j,T_j,S_j,U_j$ depend on
each other as follows. The number $T_0$ depends on $U_0$, since the existence of $T_0$ follows
from a given $U=U_0$ in Lemma \ref{directdist}. From $T_0$ we get some new large scale $S_1$ and its corresponding
new Whitney number $k_1$ (see proof of Lemma \ref{induI}). Obviously, $N_l$ depends on $S_l$.
The neighbourhood $U_1$ depends on $N_1$ since $U_1$ is defined in terms of the first return time from $U_1$ into itself is at least
$2N_l$. Then again $T_1$ depends on $U_1$ and so on. One can write this as a scheme as follows. We write $X \raw Y$ if $Y$ depends on $X$ but not the converse.
\[
f \raw S_0 \raw N_0 \raw U_0 \raw T_0 \raw S_1 \raw N_1 \raw U_1 \raw T_1 \raw S_2 \raw \ldots.
\]
Since there are no loops in this scheme, i.e. there are no two distinct
elements $X,Y$ for which both $X \raw Y$ and $Y \raw X$, there is no problem
of choosing $S_j,U_j,N_j$.

Moreover, they are independent of $r > 0$, for all $r \leq R$, for some fixed (sufficiently small) $R > 0$.
\end{Rem}

\begin{Lem}\label{Green}

Assume that $f_0$ is a Misiurewicz map. Then to any compact subset $K$ of the Fatou set $F(f_0)$ there is some $r > 0$ such that $K \subset F(f_a)$ for all $a \in \B(0,r)$.
\end{Lem}
\begin{proof}
Recall that the only Fatou components for Misiurewicz maps are those
corresponding to attracting cycles. Assume first that $K$ belongs to a
given basin of attraction. That means that in the geometrically
attracting case (where the corresponding attracting fixed point is not
super-attracting) the conjugating function $\varphi$ can be extended
to the whole basin. In the super-attracting case there is a Greens
function $\varphi$ arising from the conjugating function. In both
cases there are level lines when $|\varphi(z)|$ is constant. Since $K$ is compact there is some $\al \in \R$ so that for any $z\in K$, $|\varphi(z)| < \al$. Let $N_0 = \{z: |\varphi(z)| < \al \}$. Then $N_0$ is open and contains $K$. We have $f(N_0)=N_1 \subset N_0$. Put $N_k=f^k(N_0)$. Therefore,
\[
N_0 \supset N_1 \supset \ldots,
\]
and $\cap_k N_k$ is the attracting fixed point.
Since $\varphi=\varphi_0$ is continuous with respect to the parameter,
for some $r > 0$ the set $N_0$ moves continuously in $a$, such that
for any $a \in \B(0,r)$, putting $N_0'=\{ z : |\varphi_a(z)| < \al
\}$, we have $N_0' \supset K$ and $f_a(N_0') \subset N_0'$. Again we get a nested sequence of
sets $N_0' \supset N_1' \ldots$. The intersection $I=\cap_k N_k'$ is
an invariant topologically attracting set. It cannot intersect the Julia set since the Julia
set is topologically repelling. Hence $I$ is an invariant subset of
the Fatou set $F(f_a)$ of $f_a$, compactly contained in
$F(f_a)$. It follows that $I$ must be a fixed point. From this the lemma follows.



\end{proof}

\section{Conclusion and proof of Theorem \ref{cluster}}
We prove Theorem \ref{cluster} by induction finitely many times. Let
us start with the given Misiurewicz map $f=f_0$ (not flexible Latt\'es
map) for which $J(f)=\hat{\C}$. In the end we
will find a hyperbolic map arbitrarily close to $f$.

Choose some (sufficiently small) $r > 0$ and some $k_0$-Whitney full
dimensional disk $\B_0 \subset W$ (where $k=k_0 \leq 1/2$
from Lemma \ref{initdist}). Let us now argue
inductively. Assume that we have found solutions to the following
equation for $1 \leq k \leq l$ (if $l=0$ no solution is yet found):
\begin{equation} \label{sol}
\xi_{n_k+m_k,k}(a)-c_{k}(a)=0.
\end{equation}
Assume that (\ref{sol}) holds for all $a \in \B_l$ and for all $1 \leq k
\leq l$, and that $\B_l$ is an almost planar $k_l$-Whitney disk of radius $r_l$.
Let us now process as follows.

Consider the critical point $c_{l+1}$ which has finite order
contact (in $\B_l$). Indeed, if no more critical points would have
finite order contact, then there would be a small
ball $\B$ around any point in $\B_l$ such that all $\B$ are
Misiurewicz maps. This is impossible unless $l=2d-2$, and then we
are done.

Hence assume $l < 2d-2$. In this case, by Lemma \ref{induII} there is
a solution to $\xi_{n+m,l+1}(a)-c_{l+1}(a)=0$ for some $a \in \B_l$
such that $ddist(a,\partial \B_l) \geq r_l/2$. By Lemma \ref{induI} there is a new almost planar $k_{l+1}$-Whitney disk $\B_{l+1}
\subset \B_l$, where $\xi_{n+m,l+1}(a)-c_{l+1}(a)=0$ for all $a \in
\B_{l+1}$.

Now, we continue in the same way with $l$ replaced by $l+1$.

Since the dimension drops $1$ in each step, the set of parameters satisfying
\[
\xi_{n_k+m_k,k}(a)-c_{k}(a)=0, \quad \text{  for all $1 \leq k \leq
  l$},
\]
is a manifold of codimension $l$. Hence $\B_l$ has dimension equal to
$2d-2-l$. (In the last step, when $l=2d-2$ the set $\B_{2d-2}$ might reduce to a single point).

Recall that the parameter space of rational maps of degree $d$ up to
conjugacy by a M\"obius transformation is equal to $2d-2$.
Hence we can repeat the argument above finitely many times until every
critical point lies in the orbit of a super-attracting cycle. Hence, we find a
function $f_a$ for some $a \in W \subset \B(0,r)$ which is
hyperbolic. In fact every critical point lies in a super-attracting
cycle. Since $r > 0$ was arbitrarily small, Theorem \ref{cluster}
follows.
\comm{
The following lemma can be deduced from Lemma \ref{??} in \cite{MA} using the notion of bound period, but we give the proof anyway.
\begin{Lem}
There exists a neighbourhood $V$ around the set $\PP$ of all the $p_j$, and some $k < 1$, such that any disk $B \subset V$ satisfying the condition
\[
diam (B) \geq k \dist (B,\PP)
\]
has the following property. Let $a \in B(0,r)$.
There exists a number $n > 0$ such that $f_a^n: B \raw f_a^n(B)$ has strong bounded distortion and $diam (f_a^n(B)) \geq S$.
\end{Lem}

\begin{proof}
We argue that $n$ can be determined by being the largest integer for with the condition
\[
|\xi_n(a)-R_a^n(z)| \leq e^{-\be n}
\]
holds for all $z \in B$. In fact, during this time we can use Lemma \ref{bound-distortion}, to get distortion of $(R_a^n)'$ on $B$. It means that the set $B$ grows almost affinely. To estimate the size of $R_a^n(B)$ we

\end{proof}
}

\bibliographystyle{plain}
\bibliography{ref}

\end{document}